\documentclass[a4paper,12pt]{article}
\usepackage{amsmath}
\usepackage{amsfonts}
\usepackage{appendix}
\usepackage{amsthm}
\usepackage{amssymb}
\usepackage{enumitem}
\usepackage{fullpage}
\usepackage[all]{xy}
\usepackage{marginnote}
\usepackage{pgf,tikz,pgfplots}
\pgfplotsset{compat=1.15}
\usepackage{mathrsfs}
\usetikzlibrary{arrows}
\numberwithin{equation}{section}
\usepackage{graphicx}
\usepackage{caption}
\usepackage{subcaption}
\usepackage[english]{babel}
\usepackage{tocloft}
\usepackage[colorlinks, linkcolor=blue, citecolor=blue, urlcolor=blue]{hyperref}

\newcommand{\g}{\mathfrak{g}}
\newcommand{\n}{\mathfrak{n}}
\newcommand{\ka}{\mathfrak{k}}
\newcommand{\p}{\mathfrak{p}}

\newcommand{\C}{\mathbb{C}}
\newcommand{\R}{\mathbb{R}}

\newcommand{\N}{\mathbb{N}}
\newcommand{\tr}{\operatorname{tr}}
\newcommand{\Aut}{\operatorname{Aut}}
\newcommand{\SO}{\operatorname{SO}}
\newtheorem{theorem}{Theorem}[section]
\newtheorem{lemme}[theorem]{Lemma}

\newtheorem{prop}[theorem]{Proposition}

\newtheorem{example}[theorem]{Example}

\newcommand{\Pol}{\operatorname{Pol}}
\newcommand{\derpone}[2]{\frac{\partial #1}{\partial {#2}}}
\newcommand{\derptwo}[2]{\frac{\partial^2 #1}{\partial {#2}^2}}
\newcommand{\derptwobis}[3]{\frac{\partial^2 #1}{\partial #2 \partial #3}}
\newtheorem{remax}[theorem]{Remark}

\newtheorem{theoremIntro}{Theorem}

\DeclareMathOperator{\GL}{GL}
\DeclareMathOperator{\PSL}{PSL}
\DeclareMathOperator{\Lie}{Lie}
\DeclareMathOperator{\id}{id}
\DeclareMathOperator{\diag}{diag}
\DeclareMathOperator{\hatotimes}{\hat{\otimes}}

\title{Restricting holomorphic discrete series\\representations to a compact dual pair}
\author{Jan Frahm, Quentin Labriet}
\date{}
\begin{document}

\maketitle

\begin{abstract}
The goal of this article is to study the branching problem for a holomorphic discrete series representation of the conformal group of a simple Euclidean Jordan algebra $V$ restricted to the subgroup $\PSL_2(\R)\times\Aut(V)$ where $\Aut(V)$ denotes the compact group of automorphisms of $V$. We use a realization of the holomorphic discrete series on a space of vector-values $L^2$-functions as well as the \emph{stratified model} developed by the second author to relate the branching problem to the decomposition of certain representations of the compact group $\Aut(V)$ and to vector-valued orthogonal polynomials.
\end{abstract}

\section*{Introduction}

The restriction of an irreducible unitary representation of a Lie group $G$ to a non-compact closed subgroup $G'$ does in general not decompose into a direct sum of irreducible representations of $G'$, but rather into a direct integral. However, there are some pairs of groups $(G,G')$ and classes of irreducible unitary representations of $G$ which do indeed decompose discretely when restricted to $G'$. A detailed study of such settings was initiated by Kobayashi~\cite{Kobayashi94,Kobayashi98,Kobayashi98b,Kobayashi08}, thus providing a suitable framework for discrete branching problems.

One consequence of Kobayashi's results is that any holomorphic discrete series representation of a real semisimple group $G$ of Hermitian type decomposes discretely when restricted to a subgroup $H$ of Hermitian type whose corresponding Riemannian symmetric space embeds holomorphically into the one for $G$. If $(G,G')$ is a symmetric pair, he even provides a formula for the multiplicities occurring in the decomposition. In this paper, we study a family of subgroups $G'$ which are, in general, not symmetric, but still fall into the framework of discrete decomposability. The subgroups are of the form $G'=\PSL_2(\R)\times H$ for some compact group $H$, so that $\PSL_2(\R)$ and $H$ form a dual pair inside $G$.

In contrast to Kobayashi's multiplicity formula which is obtained by algebraic methods, our study is of a more analytic nature. Following Ding--Gross~\cite{DingGross93}, we realize the holomorphic discrete series representations of $G$ on a Hilbert space of vector-valued $L^2$-functions on a symmetric cone (assuming the group $G$ is of tube type). The key point is to use a particular set of coordinates that is adapted to the subgroup $G'$, yielding the so-called \emph{stratified model} introduced by the second author in \cite{Labriet22}. Loosely speaking, this set of coordinates separates the actions of $\PSL_2(\R)$ and $H$ and reduces the branching problem to one for the compact group $H$. It further relates the decomposition to certain (vector-valued) orthogonal polynomials on a real bounded domain. Using these orthogonal polynomials, we are able to give explicit formulas for the corresponding symmetry breaking resp. holographic operators, i.e. the intertwining operators that project onto the various discrete summands of the representation resp. embed the discrete summands into the holomorphic discrete series representation.

\paragraph{Statement of the results.}

Let us describe the results in more detail. Let $S_\pi$ be a holomorphic discrete series representation of the conformal group $G$ of a simple Jordan algebra $V$ associated to a representation $(\pi,V_\pi)$ of a maximal compact subgroup $K$ of $G$. Since $K$ and the automorphism group $L$ of the symmetric cone $\Omega$ of invertible squares in $V$ have isomorphic complexifications, we can consider $\pi$ to be a representation of $L$. Further, let $H=\Aut(V)$ be the group of automorphism of $V$, or alternatively the subgroup of $L$ fixing the identity element of the Jordan algebra $V$. The centralizer of $H$ in $G$ is isomorphic to $\PSL_2(\R)$, so the product $G'=\PSL_2(\R)\times H$ is a subgroup of $G$. We study the restriction of $S_\pi$ to $G'$.

Following \cite{DingGross93} and \cite{FrahmOlaffsonOrsted22}, we can realize $S_\pi$ on a space $L^2_\pi(\Omega)$ of $V_\pi$-valued functions on $\Omega$ which are square-integrable with respect to a certain operator-valued measure on $\Omega$. The second author introduced in \cite{Labriet22} a set of coordinates on $\Omega$ which are in this setting given by
\[ \iota:\R^+\times X \to \Omega, \quad \iota(t,v)=\frac{t}{r}(e+v), \]
where $e$ is the identity element of $V$ and $X$ is the real bounded domain in the orthogonal complement $(\R\cdot e)^\bot$ of $e$ in $V$ given by
\[ X=\left\{v\in (\R\cdot e)^\bot~|~e+v\in \Omega\right\}. \]

Using the coordinates $\iota$, we first restrict $S_\pi$ to $\PSL_2(\R)$ and show by explicit computations using the Lie algebra action that it decomposes into a direct sum of holomorphic discrete series representations $\rho_\lambda$ of parameter $\lambda>1$, realized on $L^2_\lambda(\R^+)=L^2(\R^+,t^{\lambda-1}\,dt)$, the first coordinate of $\iota$. Moreover, the multiplicity space turns out to be a space of vector valued polynomials on $X$ (see Section \ref{sec:Stratification}). This gives a natural correspondence between symmetry breaking and holographic operators and vector-valued orthogonal polynomials on $X$. Since $\PSL_2(\R)$ and $H$ commute, the latter will act on the multiplicity space, and hence on the vector-valued polynomials on $X$. This leads to our main result:

\begin{theoremIntro}[see Theorem~\ref{thm:BranchingGeneralVectorValued}]
For every holomorphic discrete series representation $S_\pi$, we have the following branching rule:
\begin{equation*}
S_\pi|_{\PSL_2(\R)\times H} \simeq {\sum_{p\in \N}}^\oplus \rho_{\alpha+2p}\otimes\big(\Pol_p(X)\otimes V_\pi|_H\big).
\end{equation*}
where $\rho_\lambda$ is the holomorphic discrete series representation of $\PSL_2(\R)$ of parameter $\lambda$, $\alpha$ is a constant depending on $\pi$, and $\Pol_p(X)$ denotes the space of homogeneous polynomials of degree $p$ on $X$ as a representation of $H$.
\end{theoremIntro}

This result reduces the branching problem to the irreducible decomposition of the tensor product $\Pol_p(X)\otimes V_\pi|_H$ which involves only finite dimensional representations. We decompose it explicitly into an orthogonal direct sum of irreducible representations of $H$ (with respect to a certain inner product induced from the inner product on $L^2_\pi(\Omega)$, see Section~\ref{sec:BranchingPSL2RxH} for details):
\[
\Pol_p(X)\otimes V_\pi|_H= \bigoplus_j F_p^j.
\]
Denote by $K_p^j:X\times X\to \operatorname{End}(V_\pi)$ the reproducing kernel of $F_p^j$, we can then explicitly describe the symmetry breaking and holographic operators as follows.

\begin{theoremIntro}[see Theorem~\ref{thm:SymBreakProduct}]
	\begin{enumerate}[label=(\alph*)]
		\item The operator $\phi_\pi^{p,j}:L^2_\pi(\Omega)\to L^2_{\alpha+2p}(\R^+)\otimes F_p^j$ defined by
		\[
			\phi_\pi^{p,j}f(t,v)=t^{-p}\int_X K_p^{j}(u,v)f(\iota(t,u)) \Delta(e+u)^{-\frac{n}{r}}~du,
		\]
		is a symmetry breaking operator between $S_\pi|_{\PSL_2(\R)\times H}$ and $ \rho_{\alpha+2p}\otimes F_p^j $.
		\item The operator $\Phi^\pi_{p,j}:L^2_{\alpha+2p}(\R^+)\otimes F_p^j\to L^2_\pi(\Omega)$ defined by
		\[
			\Phi^\pi_{p,j}f(t,v)=t^p f(t,v),
		\]
		is a holographic operator between $\rho_{\alpha+2p}\otimes F_p^{j}$ and $S_\pi|_{\PSL_2(\R)\times H}$.
	\end{enumerate}
\end{theoremIntro}

We remark that for the case $\mathfrak{g}=\mathfrak{so}(2,n)$, the subgroup $\mathfrak{g}'=\mathfrak{so}(2,1)\oplus\mathfrak{so}(n-1)$ is actually symmetric and our results agree with the more general decomposition obtained by Kobayashi in \cite{Kobayashi08}.

\paragraph{Acknowledgements.}

Both authors were supported by a research grant from the Villum Foundation (Grant No. 00025373).

\section{Preliminaries}

We recall the basic facts about Euclidean Jordan algebras and their associated groups from \cite{FarautKoranyi94} and use them to describe $L^2$-models for holomorphic discrete series representations as in \cite{DingGross93}.

\subsection{Euclidean Jordan algebras}

In this subsection, we set up the necessary notation for Jordan algebras and refer to \cite{FarautKoranyi94} for the precise definitions. 

Let $V$ be a simple Euclidean Jordan algebra of dimension $n$, rank $r$ and with unit element $e$. Let $\Omega$ be its associated symmetric cone, i.e. the interior of the set of squares in $V$. We denote by $V_\C$ the complexification of $V$ and by $T_\Omega=V+i\Omega\subset V_\C$ the corresponding tube domain. 

Write $\tr(x)$ and $\Delta(x)$ for the Jordan trace and Jordan determinant of $x\in V_\C$. Notice that $\tr(e)=r$ and $\Delta(e)=1$. We denote by $(\cdot|\cdot)$ the trace form on $V_\C$ given by $(x|y)=\tr(xy)$ for $x,y \in V_\C$. 

Further, let $L(x):V_\C\to V_\C$ be the multiplication by $x\in V_\C$. Define the quadratic representation $P(x)$, and its polarized version $P(x,y)$ for $x,y\in V_\C$ by:
\[ P(x)=2L(x)^2-L(x^2),~~~~~~P(x,y)=L(x)L(y)+L(y)L(x)-L(xy).\]
Finally, we define the box operator for $x,y\in V_\C$ by
\[x\square y=L(xy)+[L(x),L(y)].\]

Let $G=\Aut(T_\Omega)$ denote the group of biholomorphic automorphisms of $T_\Omega$. It is well-known that $G$ is a simple Lie group with trivial center. It acts transitively on $T_\Omega$ and the stabilizer $K$ of $ie$ is a maximal compact subgroup of $G$. If $\theta$ denotes the Cartan involution of $G$ which fixes $K$, we write
\[\g=\ka\oplus \p.\]
for the corresponding Cartan decomposition of the Lie algebra $\g$ of $G$.

We consider several subgroups of $G$. First, consider the subgroup of translations
\[N=\left\{\tau_u:  x\in T_\Omega \mapsto x+u\in T_\Omega~|~u\in V\right\}.\]
Further, the group 
\[L=\{g\in GL(V)~|~g\cdot \Omega\subset \Omega \},\]
acts linearly on $T_\Omega$. Then $H=K\cap L$ is a maximal compact subgroup of $L$ and it equals the automorphism group of the Jordan algebra $V$, i.e.
\[ H = \{g\in\GL(V):g(x\cdot y)=g(x)\cdot g(y)\mbox{ for all }x,y\in V\}. \]
The subgroup $P=LN$ is a maximal parabolic subgroup of $Co(V)$ with Levi factor $L$ and unipotent radical $N$. Finally, $N$ and $L$ together with the inversion
\[j:z\in T_\Omega \mapsto -z^{-1}\in T_\Omega,\]
generate the group $G$. Note that the Cartan involution $\theta$ is given by $\theta(g)=jgj^{-1}$ ($g\in G$),

Finally, we define the group
\[\bar{N}=\left\{\bar{\tau}_u=j\tau_uj^{-1} ~|~u\in V\right\}. \]
On the Lie algebra level, we have the Gelfand--Naimark decomposition
\begin{equation}
	\g = \n\oplus \mathfrak{l}\oplus \bar{\n},\label{eq:GelfandNaimark}
\end{equation}
where $\mathfrak{l}=\Lie(L)$, $\n=\Lie(N)$ and $\bar{\n}=\Lie(\bar{N})$. The Lie algebra $\mathfrak{l}$ is generated by the elements $u\square v$ for $u,v\in V$, and the Lie algebras $\n$ and $\bar{\n}$ are spanned by
$$ N_u=\left.\frac{d}{dt}\right|_{t=0}\tau_{tu}\in\n \qquad \mbox{and} \qquad \overline{N}_u=\left.\frac{d}{dt}\right|_{t=0}\bar{\tau}_u\in\bar{\n} \qquad (u\in V), $$
respectively.

Fix a Jordan frame $(c_1,\cdots, c_r)$ of $V$, and define $h_j=2L(c_j)$ so that $\sum h_j=2L(e)$. This leads to a maximal abelian subspace:
\[\mathfrak{a}=\bigoplus_{i=1}^r \R h_i\]
of $\mathfrak{l}\cap \p$ and this is also maximal abelian in $\p$. The associated root system $\Sigma(\g,\mathfrak{a})$ is of type $C_r$, and is given by
\[\left\{\frac{1}{2}(\pm\gamma_j\pm \gamma_k)~|~ 1\leq j<k\leq r \right\}\cup \left\{ \pm\gamma_j~|~1\leq j\leq r \right\} \]
with $\gamma_i(h_j)=2\delta_{ij}$, and the subsystem $\Sigma(\mathfrak{l},\mathfrak{a})$ is
$$ \left\{\frac{1}{2}(\gamma_j-\gamma_k):1\leq j\neq k\leq r\right\}. $$
We choose the positive system $\Sigma^+ (\g,\mathfrak{a})$ induced by the ordering $\gamma_r>\cdots >\gamma_1>0$ and let $\Sigma^+(\mathfrak{l},\mathfrak{a})=\Sigma^+(\g,\mathfrak{a})\cap\Sigma(\mathfrak{l},\mathfrak{a})$.

\subsection{$L^2$ model for vector valued holomorphic discrete series}

In \cite{DingGross93} an $L^2$-model for holomorphic discrete series representations of the group $G$ was constructed. We use a slight variation of this model as introduced in \cite{FrahmOlaffsonOrsted22}.

Choose a finite-dimensional irreducible representation $(\pi,V_\pi)$ of $L$. Its restricted lowest weight is given by
\[-\frac{1}{2}\sum_{i=1}^r m_i\gamma_i.\]
We also define 
\[\omega(\pi)=m_r \qquad \text{and} \qquad \alpha=\sum_{i=1}^r m_i.\]
For an element $x\in \Omega$, we set $\pi(x)= \pi(P(x))$

Notice that $L(e)=\id_V$ is contained in the center of $L$, hence by Schur's Lemma:
\[d\pi(L(e))=\nu\id_{V_\pi} =-\frac{1}{2}\alpha \id_{V_\pi}.\]
As a consequence we get for $t\in \R^+$ and $x\in \Omega$:
\begin{equation}
	\pi(tx)=\pi(P(tx))=\pi(t^2\id_V)\pi(x)=t^{-\alpha}\pi(x).\label{eq:PiRescaled}
\end{equation}

For $\omega(\pi)>\frac{2n}{r}-1$, define the operator $\Gamma_\pi$ acting on $V_\pi$ by the absolutely convergent integral
\begin{equation}
\Gamma_\pi=\int_\Omega e^{-2\tr(u)}\pi(u)^{-1}\Delta(u)^{-\frac{2n}{r}}~du.
\end{equation}
It is known that if $\pi|_H$ is irreducible then $\Gamma_\pi$ is a scalar and is equal to
\[2^{-r}\pi^{r(r-1)\frac{d}{4}}\prod_{j=1}^r \Gamma\left(m_j-\frac{n}{r}-(j-1)\frac{d}{2}\right).\]
We introduce the following Hilbert space:
\begin{equation}
L^2_\pi(\Omega):=\left\{f:\Omega\to V_\pi ~|~ \int_\Omega\langle \Gamma_\pi \pi(u^\frac{1}{2})^{-1}f(u)|\pi(u^\frac{1}{2})^{-1}f(u)\rangle \Delta(u)^{-\frac{n}{r}}~du < \infty\right\},
\end{equation}
where $u^{\frac{1}{2}}\in\Omega$ denotes the unique square root of $u\in\Omega$. On $L^2_\pi(\Omega)$ consider the following action of $G$:
\begin{align}
S_\pi(\tau_u)f(x)&=e^{-i(x|u)}f(x), &&(u\in V),\label{eq:SpiOnN}\\
S_\pi(g)f(x)&=\pi(g^*)^{-1}f(g^*x), &&(g\in L),\label{eq:SpiOnL}\\
S_\pi(j)f(x)&=\int_\Omega \mathcal{J}_\pi(u,x) f(u)\Delta(u)^{-\frac{n}{r}}~du.\label{eq:SpiOnJ}
\end{align}
Here $\mathcal{J}_\pi(u,x)$ denotes the operator-valued Bessel function associated to $\pi$ (see \cite[Definition~3.5]{DingGross93} and \cite[Section 1.6]{FrahmOlaffsonOrsted22}).

For $\omega(\pi)>\frac{2n}{r}-1$, this representation is equivalent to the holomorphic discrete series representation with highest weight space isomorphic to $V_\pi$. If $\dim(V_\pi)=1$, then we call $S_\pi$ a scalar-valued holomorphic discrete series representation. One recovers the scalar-valued case by choosing $\pi(g)=|\det(g)|^{-\frac{r\lambda}{2n}}$, then $m_1=\cdots=m_r=\lambda$, so that $\alpha=r\lambda$.

On the smooth vectors, the derived representation is given by:
\begin{align*}
dS_\pi(N_u)&=-i(x|u),\\
dS_\pi(S)&=\partial_{S^*x}-d\pi(S^*),\\
dS_\pi(\overline{N}_v)&=i(v|\mathcal{B}_\pi).
\end{align*}
Here $S^*$ denotes the adjoint of $S\in\mathfrak{l}$ with respect to the trace form and $\mathcal{B}_\pi$ denotes the vector-valued Bessel operator given by
\begin{equation}\label{eq:DefinitionBesselOperator}
(v|\mathcal{B}_\pi)=\sum_{i,j}(v|P(e_i,e_j)x)\derptwobis{}{e_i}{e_j}-2\sum_i d\pi(v\square e_i)\derpone{}{e_i}.
\end{equation}
The space of $K$-finite vectors in $L^2_\pi(\Omega)$ is the space of functions of the form
$$ x\mapsto p(x)e^{-\tr(x)} \qquad (x\in\Omega), $$
where $p\in\Pol(V,V_\pi)$, a polynomial with values in $V_\pi$.\\

We take a closer look at the case case $V=\R$ where $G\simeq\PSL_2(\R)$. Here all holomorphic discrete series representations are scalar-valued. The above discussion gives a realization $\rho_\lambda$ ($\lambda>1$) on
$$ L^2_\lambda(\R^+) = L^2(\R^+,t^{\lambda-1}\,dt) $$
by the following formulas:
\begin{align}
d\rho_\lambda\begin{pmatrix}0&1\\0&0\end{pmatrix} &= -it,\label{eq:HolDsActionPSL2a}\\
d\rho_\lambda\begin{pmatrix}1&0\\0&-1\end{pmatrix} &= 2t\frac{d}{d t}+\lambda,\label{eq:HolDsActionPSL2b}\\
d\rho_\lambda\begin{pmatrix}0&0\\-1&0\end{pmatrix} &= i\mathcal{B}_\lambda^\R = i\left(t\frac{d^2}{dt^2}+\lambda\frac{d}{dt}\right),\label{eq:HolDsActionPSL2c}
\end{align}

\section{Restriction of a holomorphic discrete series}

We consider the restriction of a representation $S_\pi$ of the holomorphic discrete series of $G$ to a non-compact subgroup $G'$ of the form $\PSL_2(\R)\times H$. Here, $H$ is the automorphism group of the cone $\Omega$, and its centralizer in $G$ turns out to be the image of the map
$$ \psi:\PSL_2(\R)\to G, \quad \psi\begin{pmatrix}a&b\\c&d\end{pmatrix}z = \frac{az+b}{cz+d} \qquad (z\in T_\Omega). $$
To do so, we use a specific change of variables introduced in \cite{Labriet22} and transfer the $L^2$-model for the representation into the stratified model.
 
\subsection{Stratification}\label{sec:Stratification}

In \cite{Labriet22} a stratification map was introduced to study the restriction of holomorphic discrete series realized on $L^2$-functions on $\Omega$ to the conformal group of a Jordan subalgebra $V_1$ of $V$. The stratification map $\iota$ is a diffeomorphism from $\Omega_1\times X$ to $\Omega$, where $\Omega_1$ is the symmetric cone in $V_1$ and $X$ is a certain real bounded domain. In this paper, we focus on the special case where $V_1=\R\cdot e$.

In this situation, the stratification space is the domain
\[X=\left\{v\in (\R\cdot e)^\bot~|~e+v\in \Omega\right\} \subseteq (\R\cdot e)^\perp,\]
and the stratification map $\iota:\R^+\times X \to \Omega$ becomes
\[\iota(t,v)=\frac{t}{r}(e+v).\]
Choosing an orthonormal basis $\{e_0,\ldots,e_{n-1}\}$ of $V$ such that $e_0=r^{-\frac{1}{2}}e$ and $(e_1,\cdots,e_{n-1})$ is an orthonormal basis for $(\R\cdot e )^\bot$, the inverse of $\iota$ can be expressed as
\[\iota^{-1}(x)=(rx_0,x_0^{-1}x'),\]
where $x=x_0 e+\sum x_i e_i$ and $x'=\sum x_i e_i$. If $dv$ denotes Lebesgue measure on $X$ normalized by the trace form $(\cdot|\cdot)$, then the following integral formula holds:
\begin{equation}
	\int_\Omega f(x)\,dx = r^{\frac{1}{2}-n}\int_X\int_{\R^+} f(\iota(t,v))t^{n-1}\,dt\,dv.\label{eq:StratificationIntFormula}
\end{equation}

We first study how the operator-valued gamma function $\Gamma_\pi$ that is used to define the invariant inner product of $L^2_\pi(\Omega)$ behaves with respect to these coordinates.

\begin{lemme}
The operator $\Gamma_\pi$ factors as
\begin{equation}
\Gamma_\pi=\Gamma_\alpha\Gamma_{\pi,X},
\end{equation}
where $\Gamma_\alpha=\frac{\Gamma(\alpha-n)}{r^{\alpha-n-\frac{1}{2}}2^{\alpha-n}}$ and $\Gamma_{\pi,X}$ denotes the operator defined by:
\begin{equation}
\Gamma_{\pi,X}=\int_X\pi(e+v)^{-1}\Delta(e+v)^{-\frac{2n}{r}}~dv.
\end{equation}
Furthermore, if $\pi|_H$ is irreducible then $\Gamma_{\pi,X}$ is a scalar.
\end{lemme}

\begin{proof}
Using \eqref{eq:PiRescaled} and \eqref{eq:StratificationIntFormula} one gets:
\begin{align*}
\Gamma_\pi&=\int_\Omega e^{-2\tr(u)}\pi(u)^{-1}\Delta(u)^{-\frac{2n}{r}}~du\\
&=\frac{1}{r^{\alpha-n-\frac{1}{2}}}\left(\int_{\R^+} e^{-2t}t^{\alpha-n-1}~dt\right)
\left(\int_X\pi(e+v)^{-1}\Delta(e+v)^{-\frac{2n}{r}}~dv\right).
\end{align*}
The last statement is direct since $\Gamma_\pi$ already is a scalar in this case. 
\end{proof}
 
This Lemma allows us to identify $L^2_\pi(\Omega)$ with a space of functions on $\R_+\times X$.

\begin{prop}
The pullback $\iota^*$ is a scalar multiple of a unitary map between $L^2_\pi(\Omega)$ and $L^2_\alpha(\R^+)\hatotimes L^2_\pi(X)$ where:
\begin{multline*}
	L^2_\pi(X) = \Big\{f:X\to V_\pi\Big|\int_X\langle\Gamma_{\pi,X}\pi((e+v)^\frac{1}{2})^{-1}f(v),\pi((e+v)^\frac{1}{2})^{-1}f(v)\rangle \\
	\times\Delta(e+v)^{-\frac{n}{r}}~dv<\infty\Big\}.
\end{multline*}
More precisely:
\[\|f\|^2_{L^2_\pi(\Omega)}=\frac{\Gamma_\alpha}{r^{\alpha-\frac{1}{2}}}\|f\circ\iota\|_{L^2(\R^+,t^{\alpha-1}~dt)\hatotimes L^2_\pi(X)}^2\]
\end{prop}

\begin{proof}
This is a direct computation similar to the one for the previous lemma.
\end{proof}

We now study how the isomorphism between $L^2_\pi(\Omega)$ and $L^2(\R^+,t^{\alpha-1}\,dt)\hatotimes L^2_\pi(X)$ can be used to decompose the restriction of $S_\pi$ to $\PSL_2(\R)\times H$. Note that $H$ is contained in $L$, so its action in $S_\pi$ is given by \eqref{eq:SpiOnL}:
\begin{itemize}
	\item For $k\in H$, we get:
	\begin{equation}
		S_\pi(k)f(t,v)=\pi(k)f(t,k^{-1}v).\label{eq:StratfiedActionH}
	\end{equation}
\end{itemize}

On the factor $\PSL_2(\R)$, the group action involves the complicated operator-valued Bessel function, so we use the action of the Lie algebra of $\PSL_2(\R)$ instead. Note that the Lie algebra of $\PSL_2(\R)$ decomposes according to the Gelfand--Naimark decomposition \eqref{eq:GelfandNaimark} as
\[ \Lie(\PSL_2(\R)) = \n_1\oplus \R L(e)\oplus \bar{\n}_1,\]
with $\n_1\subset \n$ and $\bar{\n}_1\subset \bar{\n}$ corresponding to the embedding $\R e\subseteq V$. This gives the following actions in the stratified model:
\begin{itemize}
\item For the translations of $\PSL_2(\R)$:
\begin{equation}
	S_\pi(\psi\begin{pmatrix}1&b\\0&1\end{pmatrix})f(t,v)=e^{-itb}f(t,v).\label{eq:StratifiedActionPSL2a}
\end{equation}
\item The matrix $g=\diag(a,a^{-1})\in\PSL_2(\R)$ acts via scalar multiplication by $a^2$ on $\Omega$, hence we have $\pi(\diag(a,a^{-1}))=\pi(a^2\cdot\id)=a^{-\alpha}$, so:
\begin{equation}
	S_\pi(\psi\begin{pmatrix} a&0\\0&a^{-1}\end{pmatrix})f(t,v)=a^{\alpha}f(a^2t,v).\label{eq:StratifiedActionPSL2b}
\end{equation}
\end{itemize}

Finally, the $\bar{\n}_1$ action is given by the following:

\begin{prop}\label{prop:BesselOperator(t,v)VectorValued}
The action of $\bar{\n}_1\subseteq\mathfrak{sl}_2(\R)$ in the stratified model is given by
\[ dS_\pi(d\psi\begin{pmatrix}0&0\\1&0\end{pmatrix})=i\left(\mathcal{B}_\alpha^\R +t^{-1}D_\pi\right), \]
where $\mathcal{B}_\alpha^\R$ denotes the Bessel operator for the one dimensional Jordan algebra $\R$ (see \eqref{eq:HolDsActionPSL2c}) and $D_\pi$ is the following second order differential operator in the variable $v\in X$:
\[
D_\pi=r\sum_{i=1}^{n-1}\derptwo{ }{v_i}+\sum_{1\leq i,j \leq n-1}\!\!\!\left(r(e_ie_j|v)-v_iv_j\right)\derptwobis{}{v_i}{v_j}-\alpha\sum_{i=1}^{n-1}v_i\derpone{}{v_i}-2r\sum_{i\geq 1} d\pi(L(e_i))\derpone{}{v_i}.
\]
\end{prop}

\begin{proof}
For a smooth function $f:\Omega\to V_\pi$ we have the formula
\begin{equation}
dS_\pi\begin{pmatrix}0&0\\1&0\end{pmatrix} f(x)=i\left(\sum_{i,j}(e|P(e_i,e_j)x)\derptwobis{f}{e_i}{e_j}-2\sum_i d\pi(e\square e_i)\derpone{f}{e_i}\right).
\end{equation}
In what follows we consider $x=x_0 e+\sum x_i e_i$ hence:
\[\derpone{f}{e_0}=\frac{1}{r^\frac{1}{2}}\derpone{f}{x_0} \quad \text{and} \quad \derpone{f}{e_i}=\derpone{f}{x_i}.\]
The second order part of the Bessel operator, denoted $\mathcal{B}_0$, gives:
\begin{multline}\label{eq:Second_order_Bessel}
(e|\mathcal{B}_0f(x))\\=(e|P(e_0,e_0)x)\derptwo{f}{e_0}+2\sum_{i=1}^{n-1}(e|P(e_i,e_0)x)\derptwobis{f}{e_0}{e_i}+\sum_{i,j\geq 1}(e|P(e_i,e_j)x)\derptwobis{f}{e_i}{e_j}.
\end{multline}
Hence, using the previous remark and the fact that $(e|e)=r$ we get:
\[
(e|\mathcal{B}_0f(x))\\= \frac{x_0}{r}\derptwo{f}{x_0}+\frac{2}{r}\sum_{i=1}^{n-1}x_i\derptwobis{f}{x_0}{x_i}+\sum_{i,j\geq 1}(e|P(e_i,e_j)x)\derptwobis{f}{x_i}{x_j}.
\]

Next we compute the derivatives for the variables $(t,v)$ with respect to the variables $x_i$ and we get:
\begin{align*}
	\derpone{t}{x_0} &= r, & \derpone{t}{x_i} &= 0 && (i\geq1),\\
	\derpone{v_i}{x_0} &= -rt^{-1}v_i, & \derpone{v_i}{x_j} &= \delta_{ij}rt^{-1} && (i,j\geq 1).
\end{align*}
Using the chain rule this gives:
\begin{align*}
	\derpone{f}{x_0} &= r\derpone{f}{t}-rt^{-1}\sum_{i=1}^{n-1} v_i\derpone{f}{v_i},\\
	\derpone{f}{x_i} &= rt^{-1}\derpone{f}{v_i},\\
	\derptwo{f}{x_0} &= r^2\left(\derptwo{f}{t}+2t^{-2}\sum_{i=1}^{n-1}v_i\derpone{f}{v_i}-2t^{-1}\sum_{i=1}^{n-1}v_i\derptwobis{f}{v_i}{t}+t^{-2}\sum_{1\leq i,j\leq n-1}v_iv_j\derptwobis{f}{v_i}{v_j}\right),\\
	\derptwobis{f}{x_0}{x_j} &= r^2\left(t^{-1}\derptwobis{f}{v_j}{t}-t^{-2}\derpone{f}{v_j}-t^{-2}\sum_{i=1}^{n-1}v_i\derptwobis{f}{v_i}{v_j}\right),\\
	\derptwobis{f}{x_i}{x_j} &= r^2t^{-2}\derptwobis{f}{v_i}{v_j}.
\end{align*}
Plugging these derivatives into \eqref{eq:Second_order_Bessel} one gets 
\[
(e|\mathcal{B}_0f(x))=t\derptwo{f}{t}+rt^{-1}\sum_{i,j}(P(e_i,e_j)(e+v)|e)\derptwobis{f}{v_i}{v_j}-t^{-1}\sum_{1\leq i,j\leq n-1}v_iv_j\derptwobis{f}{v_i}{v_j}.
\]
Finally, we focus our attention on the first order part of the Bessel operator and we get in the coordinates $(t,v)$:
\begin{align*}
\sum_{i\geq 0} d\pi(e\square e_i)\derpone{f}{e_i}&=d\pi(L(e_0))\derpone{f}{e_0}+\sum_{i\geq 1} d\pi(e\square e_i)\derpone{f}{e_i}\\
&=-\frac{\alpha}{2r}\derpone{f}{e_0}+\sum_{i\geq 1} d\pi(e\square e_i)\derpone{f}{e_i}\\
&=-\frac{\alpha}{2}\left(\derpone{f}{t}-t^{-1}\sum_{i=1}^{n-1}v_i\derpone{f}{v_i}\right)+rt^{-1}\sum_{i\geq 1} d\pi(L(e_i))\derpone{f}{v_i}.
\end{align*}
Adding the contribution of the second order part finishes the proof.
\end{proof}

\subsection{Branching to $\PSL_2(\R)$}

Now we consider the branching to the subgroup $\PSL_2(\R)$, and we first look at the Casimir operator
$$ C = H^2+2(XY+YX) = H^2+2H+4YX, $$
where $H,X,Y$ denotes the standard basis of $\mathfrak{sl}_2(\R)$ satisfying
$$ [H,X] = 2X, \quad [H,Y] = -2Y, \quad [X,Y]=H. $$

\begin{prop}
The Casimir operator of $\mathfrak{sl}_2(\R)$ acts in the stratified model by
\begin{equation}
dS_\pi(d\psi(C))=\alpha(\alpha-1)-4D_\pi
\end{equation}
In particular, this implies that $D_\pi$ has a self-adjoint extension to $L^2_\pi(X)$.
\end{prop}

\begin{proof}
This is a direct consequence of the formulas for the derived representation and Proposition~\ref{prop:BesselOperator(t,v)VectorValued}.
\end{proof}

The operator $D_\pi$ can be written as
\begin{equation}
	D_\pi=r\Delta+\Psi_\pi-E(E+\alpha-1),
\end{equation}
where $\Delta=\sum_{i=1}^{n-1}\derptwo{}{v_i}$ denotes the Laplacian and $E=\sum_i v_i\derpone{}{v_i}$ the Euler operator on $(\R\cdot e)^\bot$, and $\Psi_\pi$ is the differential operator defined by:
\begin{equation}\label{eq:DefinitionPsi}
\Psi_\pi =\sum_{1\leq i,j\leq n-1}r(e_ie_j|v)\derptwobis{}{v_i}{v_j}-2r\sum_{i\geq 1}d\pi(L(e_i))\derpone{}{v_i}.
\end{equation}
 
\begin{prop}\label{prop:DefinitionWp}
The operator $D_\pi$ acts on the space $\Pol(X,V_\pi)$ of restrictions of polynomials on $(\R\cdot e)^\perp$ with values in $V_\pi$ to the open subset $X$, and the subspace:
\[W^\pi_p=\{P\in \Pol(X,V_\pi)|\deg(P)\leq p,(P|Q)_{L^2_\pi(X)}=0\text{ if }\deg(Q)<p\},\]
is an eigenspace for $D_\pi$ with eigenvalue $-p(p+\alpha-1)$.
\end{prop}

\begin{proof}
Since $D_\pi$ is a differential operator with polynomial coefficients, it clearly acts on $\Pol(X,V_\pi)$. It is also immediate from the expression for $D_\pi$ that $\deg(D_\pi P)\leq\deg(P)$ for all $P\in\Pol(X,V_\pi)$. Since $D_\pi$ is self-adjoint on $L^2_\pi(X)$, it follows that each $W_p^\pi$ ($p\geq0$) is invariant under $D_\pi$. Moreover, as a self-adjoint operator $D_\pi$ is diagonalizable on $W_p^\pi$. Let $P\in W_p^\pi$ be an eigenvector with eigenvalue $\mu$ and write $P=P_1+P_2$ with $P_1$ homogeneous of degree $p$ and $\deg(P_2)<p$. Then the highest order term in $D_\pi P$ is
\[-E(E-\alpha-1)P_1=-p(p+\alpha-1)P_1,\]
so $\mu=-p(p-1+\alpha)$. 
\end{proof}

Finally, this gives the following branching law, symmetry breaking and holographic operators:

\begin{theorem}\label{thm:BranchingSL2VectorValued}
Let $S_\pi$ a holomorphic discrete series representation of $G$.
\begin{enumerate}[label=(\alph*)]
\item\label{thm:BranchingSL2VectorValued1} The restriction of $S_\pi$ to the subgroup $\PSL_2(\R)$ decomposes as
\[S_\pi|_{\PSL_2(\R)}\simeq {\sum_{p\geq 0}}^\oplus\bigg[\binom{n+p-2}{n-2}\dim V_\pi\bigg]\cdot\rho_{\alpha+2p}.\]
\item\label{thm:BranchingSL2VectorValued2} For any $P\in \Pol(X,V_\pi)$ of degree $p$, the operator
\begin{multline*}
	\phi_\pi^{p}(P):L^2_{\alpha}(\R^+)\hatotimes L^2_\pi(X)\to L^2_{\alpha+2p}(\R^+),\\	
	\phi_\pi^{p}(P)f(t)=t^{-p}\int_X \langle \Gamma_{\pi,X} \pi((e+v)^\frac{1}{2})^{-1}f(t,v),\pi((e+v)^\frac{1}{2})^{-1}P(v)\rangle \Delta(e+v)^{-\frac{n}{r}}~dv.
\end{multline*}
is a symmetry breaking operator between $S_\pi$ and $\rho_{\alpha+2p}$ if and only if $P\in W^\pi_p$.
\item\label{thm:BranchingSL2VectorValued3} For any $P\in \Pol(X,V_\pi)$ of degree $p$, the operator
\begin{equation*}
	\Phi^\pi_p(P):L^2_{\alpha+2p}(\R^+) \to L^2_{\alpha}(\R^+)\hatotimes L^2_\pi(X), \quad \Phi^\pi_p(P)f(t)=t^pf(t)P(v).
\end{equation*}
is a holographic operator between $\rho_{\alpha+2p}$ and $S_\pi$ if and only if $P\in W^\pi_p$.
\end{enumerate}
\end{theorem}

\begin{proof}
Since $W_p^\pi$ is an eigenspace for $D_\pi$, the space $L^2_\alpha(\R^+)\otimes W_p^\pi$ is an eigenspace for the action of the Casimir element $dS_\pi(C)$ of $\PSL_2(\R)$. Hence, $L^2_\alpha(\R^+)\otimes W_p^\pi$ is a subrepresentation of $S_\pi|_{\PSL_2(\R)}$. By \eqref{eq:StratifiedActionPSL2a}, \eqref{eq:StratifiedActionPSL2b} and Proposition~\ref{prop:BesselOperator(t,v)VectorValued}, $\PSL_2(\R)$ only acts in the variable $t$ on each subrepresentation $L^2_\alpha(\R^+)\otimes W_p^\pi$, so for fixed $P\in W_p^\pi$ we consider the map $\Phi^\pi_p(P):L^2_{\alpha+2p}(\R^+) \to L^2_{\alpha}(\R^+)\hatotimes L^2_\pi(X)$ from \ref{thm:BranchingSL2VectorValued3}. A short computation using \eqref{eq:StratifiedActionPSL2a}, \eqref{eq:StratifiedActionPSL2b} and Proposition~\ref{prop:BesselOperator(t,v)VectorValued} as well as \eqref{eq:HolDsActionPSL2a}, \eqref{eq:HolDsActionPSL2b} and \eqref{eq:HolDsActionPSL2c} shows that this map is intertwining for $\rho_{\alpha+2p}$ and $S_\pi|_{\PSL_2(\R)}$ if and only if $P\in W^\pi_p$, so \ref{thm:BranchingSL2VectorValued3} follows. Since $\phi_\pi^p(P)$ is the adjoint of $\Phi_\pi^p(P)$, this also shows \ref{thm:BranchingSL2VectorValued2}.

Finally, the $K$-finite vectors in $L^2_\pi(\Omega)$ are of the form $P(x)e^{-\tr(x)}$ with $P$ a polynomials in $\Pol(\Omega,V_\pi)$. Using the stratification map, the $K$-finite vectors becomes $Q(t,tv)e^{-t}$ where $Q$ is a polynomial on $\R\times(\R\cdot e)^\perp$. On $L^2_{\alpha+2p}(\R^+)$ the $(K\cap\PSL_2(\R))$-finite vectors are of the form $R(t)e^{-t}$ with $R$ a polynomial on $\R$, so the $(K\cap\PSL_2(\R))$-finite vectors on each $L^2_\alpha(\R^+)\hatotimes W_p^\pi$ are of the form $t^pQ(t,v)e^{-t}=Q(t,tv)e^{-t}$ with $Q$ a polynomial in $\C[t]\otimes W_p^\pi$. This shows that the sum of all the images of $\Phi_p^\pi(P)$, $P\in W_p^\pi$, $p\geq0$, contains the $K$-finite vectors of $S_\pi$, hence it is dense in $L^2_\pi(\Omega)$. Together with the fact that $\dim(W_p^\pi)=\dim(V_\pi)=\binom{n+k-2}{n-2}$ this shows \ref{thm:BranchingSL2VectorValued1}.
\end{proof}

\subsection{Branching to $\PSL_2(\R)\times H$}\label{sec:BranchingPSL2RxH}

Recall that the group $H$ acts on $L^2_\pi(X)$ by \eqref{eq:StratfiedActionH}.

\begin{lemme}
The operator $D_\pi$ commutes with the action of $H$ on $L^2_\pi(X)$.
\end{lemme}

\begin{proof}
Since $H$ and $\PSL_2(\R)$ commute in $G$, their actions $S_\pi(k)$ ($k\in H$) and $dS_\pi(d\psi(T))$ ($T\in\mathfrak{sl}_2(\R)$) commute as well.
\end{proof}

It follows that the space $W_p^\pi$ is a representation of the compact group $H$. Since $W_p^\pi$ is difficult to work with, we relate it to the space $\Pol_p(X,V_\pi)$ of homogeneous polynomials  of degree $p$ with values in $V_\pi$ on which $H$ acts by
$$ (k\cdot P)(v) = \pi(k)P(k^{-1}v) \qquad (k\in H, P\in\Pol_p(X,V_\pi),v\in X). $$
Then any polynomial $P\in\Pol(X,V_\pi)$ of degree $p$ can be decomposed into $P=\sum_{i=0}^p P_i$ with $P_i\in \Pol_i(X,V_\pi)$ its homogeneous part of degree $i$.

\begin{prop}
The map $T:W_p^\pi \to \Pol_p(X,V_\pi)$ defined by $T(P)=P_p$ is an $H$ intertwining isomorphism. For $Q\in \Pol_p(X,V_\pi)$, the polynomial $P=T^{-1}(Q)$ is uniquely determined by the following recursion formula:
\[
\Delta P_{i+2}+\Psi_\pi P_{i+1}=\left( i(i+\alpha-1)-p(p+\alpha-1)\right) P_i \qquad (0\leq i\leq p-1),
\]  
where $\Delta$ is the Laplacian on $X$ and $\Psi_\pi$ the operator defined in \eqref{eq:DefinitionPsi}.
\end{prop}

\begin{proof}
It is clear that $T$ is an intertwining operator since the action of $H$ is given by the same formula on $W_p^\pi$ and $\Pol_p(X,V_\pi)$ and this formula preserved the degree of homogeneity. The map $T$ is injective since $T(P)=0$ implies $\deg(P)<p$, and every $P\in W_p^\pi$ is orthogonal to all polynomials of degree $<p$. Furthermore we have $\dim (W_p^\pi)=\dim(\Pol_p(X,V_\pi))$ so $T$ is a bijection.

From Proposition \ref{prop:DefinitionWp} we know that $P=\sum_{i=1}^p P_i\in W_p^\pi$ is equivalent to $D_\pi P=-p(p+\alpha-1)P$. Recalling that $D_\pi=\Delta_v+\Psi_\pi-E(E+\alpha-1)$, this is equivalent to the claimed recursion formula.
\end{proof}

By the previous result, we have the following isomorphism of $H$-representations:
$$ W_p^\pi \simeq \Pol_p(X,V_\pi)\simeq \Pol_p(X)\otimes V_\pi. $$
So to decompose $W_p^\pi$ into irreducible representations of $H$ one has to decompose $\Pol_p(X)$ and $V_\pi|_H$ into irreducible representations and then decompose the corresponding tensor products. This implies the following branching law:

\begin{theorem}\label{thm:BranchingGeneralVectorValued}
For every holomorphic discrete series representation $S_\pi$, we have the following branching rule:
\begin{equation*}
S_\pi|_{\PSL_2(\R)\times H} \simeq {\sum_{p\in \N}}^\oplus \rho_{\alpha+2p}\otimes\big(\Pol_p(X)\otimes V_\pi\big).
\end{equation*}
\end{theorem}

\begin{example}
	Since $\Pol_p(X)=S^p(U)$, where $U^*=(\R\cdot e)_\C^\perp$, the decomposition of $\Pol_p(X)$ is the classical problem of decomposing $S^p(U)$ for the following representation $U$ of $H$:
	\begin{enumerate}
		\item $V=\operatorname{Sym}(n,\R)$. Here $\mathfrak{h}_\C=\mathfrak{so}(n,\C)$ is acting on $V=\operatorname{Sym}(n,\C)$ by conjugation and  $U^*=\{X\in\operatorname{Sym}(n,\C):\tr(X)=0\}$ which can be identified with $S^2_0(\C^n)$, the trace-free part of the second symmetric power of the standard representation of $\mathfrak{so}(n)$ on $\C^n$. This representation is irreducible and self-dual, so $U=S^2_0(\C^n)$. Its highest weight in terms of the fundamental weights $\omega_1,\ldots,\omega_{\lfloor\frac{n}{2}\rfloor}$ is $2\omega_1$.
		\item $V=\operatorname{Herm}(n,\C)$. Here $\mathfrak{h}_\C=\mathfrak{sl}(n,\C)$ is acting on $V=M(n\times n,\C)$ by conjugation and $U^*=\{X\in M(n\times n,\C):\tr(X)=0\}$ which can be identified with $(\C^n\otimes(\C^n)^*)_0$, the trace-free part of the tensor product of the standard representation of $\mathfrak{sl}(n,\C)$ on $\C^n$ and its dual. Clearly, $\C^n\otimes(\C^n)^*$ is self-dual, so $U=(\C^n\otimes(\C^n)^*)_0$. This representation is irreducible and its highest weight in terms of the fundamental weights $\omega_1,\ldots,\omega_{n-1}$ is $\omega_1+\omega_{n-1}$.
		\item $V=\operatorname{Herm}(n,\mathbb{H})$. Here $\mathfrak{h}_\C=\mathfrak{sp}(n,\C)$ is acting on $V_\C=\{X\in\operatorname{Skew}(2n,\mathbb{C})\}$ by conjugation, so $U^*$ can be identified with $(\C^{2n}\otimes(\C^{2n})^*)_0$, the trace-free part of the tensor product of the standard representation of $\mathfrak{sp}(n,\C)$ on $\C^{2n}$ and its dual. Since the standard representation is self-dual, we find $U=S^2_0(\C^{2n})$. This representation is irreducible and its highest weight in terms of the fundamental weights $\omega_1,\ldots,\omega_n$ is $2\omega_1$.
		\item $V=\operatorname{Herm}(3,\mathbb{O})$. Here $\mathfrak{h}_\C=\mathfrak{f}_4(\C)$ is acting on the $26$-dimensional space $U^*$. Since this is the smallest dimension of a non-trivial irreducible representation, and there is precisely one irreducible representation of this dimension, we find that $U$ is the fundamental representation of $\mathfrak{f}_4(\C)$ with highest weight $\omega_1$ in terms of the fundamental weights $\omega_1,\omega_2,\omega_3,\omega_4$.
		\item $V=\R\times\R^{n-1}$. Here $\mathfrak{h}_\C=\mathfrak{so}(n-1,\C)$ is acting on $U^*=\{0\}\times\C^{n-1}=\C^{n-1}$ by the standard representation which is self-dual, so $U^*$ can be identified with $\C^{n-1}$. This representation is irreducible and its highest weight in terms of the fundamental weights $\omega_1,\ldots,\omega_{\lfloor\frac{n-1}{2}\rfloor}$ is $\omega_1$.
	\end{enumerate}
\end{example}

To also state formulas for the symmetry breaking and holographic operators, we decompose $W_p^\pi$ into irreducible representations $F_p^j$:
\[W_p^\pi=\bigoplus_j F_p^j.\]
We also denote by $F_p^j$ the action of $H$ on the vector space $F_p^j$, more precisely:
\[
F_p^j(k)P(v)=\pi(k)P(k^{-1}\cdot v).
\]
Note that this decomposition is not unique since $F_p^j$ might occur with higher multiplicity in $W_p^\pi$.

For every $v\in X$ and $\xi\in V_\pi$, the linear form $F_p^j\to\C,\,P\mapsto\langle P(v),\xi\rangle$ is represented by a vector $K_v\xi\in F_p^j$, i.e. $\langle P(v),\xi\rangle=\langle P,K_v\xi\rangle$. We write $K(u,v)\xi=(K_v\xi)(u)$. Since $K(u,v)\xi$ is linear in $\xi$, we have $K_p^j(u,v)=K(u,v)\in\operatorname{End}(V_\pi)$. This function $K_p^j:X\times X\to\operatorname{End}(V_\pi)$ satisfies
\begin{equation}\label{eq:InterwiningPropertyKernel}
	K_p^j(k\cdot u,k\cdot v)=\pi(k)K_p^j(u,v)\pi(k)^{-1} \qquad (u,v\in X,k\in H).
\end{equation}

\begin{theorem}\label{thm:SymBreakProduct}
	\begin{enumerate}[label=(\alph*)]
		\item The operator $\phi_\pi^{p,j}:L^2_\pi(\Omega)\to L^2_{\alpha+2p}(\R^+)\otimes F_p^j$ defined by
		\[
			\phi_\pi^{p,j}f(t,v)=t^{-p}\int_X K_p^{j}(u,v)f(t,u) \Delta(e+u)^{-\frac{n}{r}}~du,
		\]
		is a symmetry breaking operator between $S_\pi|_{\PSL_2(\R)\times H}$ and $ \rho_{\alpha+2p}\otimes F_p^j $.
		\item The operator $\Phi^\pi_{p,j}:L^2_{\alpha+2p}(\R^+)\otimes F_p^j\to L^2_\pi(\Omega)$ defined by
		\[
			\Phi^\pi_{p,j}f(t,v)=t^p f(t,v),
		\]
		is a holographic operator between $\rho_{\alpha+2p}\otimes F_p^{j}$ and $S_\pi|_{\PSL_2(\R)\times H}$.
	\end{enumerate}
\end{theorem}

\begin{proof}
The intertwining property for the $\PSL_2(\R)$ action is a consequence of Theorem \ref{thm:BranchingSL2VectorValued}. The intertwining property for $H$ is a consequence of formula \eqref{eq:InterwiningPropertyKernel}.
\end{proof}

\subsection{The case $\mathfrak{g}=\mathfrak{so}(2,n+1)$}

We conclude by considering the special case where $V$ is a Euclidean Jordan algebra of rank $2$. More precisely, we have $V=\R\times \R^n$, and the Jordan multiplication is given by:
\[
(x,u)\cdot (y,v)=(xy+B(u,v),xv+yu),
\] 
where $B(u,v)$ denotes the usual inner product on $\R^n$. %To avoid specific cases we assume $n\geq 3$. 
The inner product on the Euclidean Jordan algebra $V$ is given by:
\[(x|y)=\tr(xy)=2\sum x_i y_i,\]
for $x=\sum x_i f_i$, $y=\sum y_i f_i$ ($0\leq i\leq n$) where $\{f_i\}$ is the canonical basis on $\R\times \R^n$ which is not an orthonormal basis.
The symmetric cone is
\[\Omega=\{x\in V\ |\ Q_{1,n}(x)> 0,\ x_1>0\},\]
hence we have $L\simeq\SO_0(1,n)$ and $H=\SO(n)$. This leads to:
\[X=\{(0,v)\ |\ v\in\R^n,\ \|v\|<1 \}=B^n.\]
In this situation, it is well known that as an $\SO(n)$ representation
\[
\Pol_p (X)\simeq \bigoplus_{j=1}^{\lfloor p/2\rfloor} \mathcal{H}_{p-2j}^n,
\]
where $\mathcal{H}_p^n$ denotes the irreducible representation of $\SO(n)$ on the space of harmonic polynomials of degree $p$ in $n$ variables. Thus, to find the explicit abstract branching law one needs to deal with tensor products of the form
\[
\mathcal{H}_{p-2j}^n\otimes \pi.
\]
This can be decomposed explicitly for $\pi$ being a representation of the form $\mathcal{H}_k^{n+1}$ or a fundamental representation using the classical branching rules and the results in \cite{HoweTanWillenbring05}.

Since $f_i\cdot f_j=0$ for $i,j\geq 1$ and $i\neq j$, and $f_i^2=e$, the operator $D_\pi$ is explicitly given by:
\[
D_\pi=\sum_{i=1}^n\derptwo{ }{x_i}-\sum_{1\leq i,j\leq n}x_ix_j\derptwobis{ }{x_i}{x_j}-\alpha\sum_{i=1}^n x_i\derpone{ }{x_i}-2\sum_{i=1}^nd\pi(L(f_i))\derpone{ }{x_i}.
\]
Notice that if we consider the case where $\pi$ is a one dimensional character, i.e. $S_\pi$ is a scalar-valued holomorphic discrete series representation, then we recover the operator considered in \cite{Labriet22} for the pair $(SO_0(2,n),SO_0(2,n-p))$ for the special case $p=n-1$.

\begin{remax}
In this case, the pair $(\mathfrak{g},\mathfrak{g}')=(\mathfrak{so}(2,n+1),\mathfrak{so}(2,1)\oplus\mathfrak{so}(n))$ is a symmetric pair with the isomorphism $\mathfrak{so}(2,1)\simeq\mathfrak{sl}(2,\R)$. Hence, the abstract branching law is a special case of the more general result of Kobayashi in \cite[Lemma 8.8]{Kobayashi08}. His proof relies on the Hua--Kostant--Schmid formula for the space of polynomials $\Pol((e^\bot)^\C)$ under the action of $\mathfrak{k}\cap\mathfrak{g}'=\mathfrak{so}(2)\oplus\mathfrak{so}(n)$. In our approach, we took care of the $\mathfrak{so}(2)$-action before studying the action of $\mathfrak{so}(n)$ on $\Pol(e^\bot)$ and this corresponds to the grading of this space by homogeneous polynomials of a fixed degree.
\end{remax}

%\bibliographystyle{alpha}
%\bibliography{./Biblio}

\end{document}